\documentclass[12pt]{amsart}

\usepackage{amsmath, amsthm, amssymb, enumerate, amscd}

\theoremstyle{plain}
\newtheorem{theorem}{Theorem}[section]
\newtheorem{proposition}[theorem]{Proposition}
\newtheorem{corollary}[theorem]{Corollary}

\theoremstyle{definition}
\newtheorem{definition}[theorem]{Definition}
\newtheorem{definition_basic_facts}[theorem]{Definition and basic facts}
\newtheorem{remark}[theorem]{Remark}

\newtheorem{example}[theorem]{Example}

\newcommand{\F}{{\mathbb{F}}}
\newcommand{\M}{{\mathbb{M}}}
\newcommand{\PM}{{\mathbb{PM}}}
\newcommand{\PW}{{\mathbb{PW}}}
\newcommand{\CC}{{\mathcal{C}}}
\newcommand{\DC}{{\mathcal{D}}}
\newcommand{\am}{{\mathbf{a}}}
\newcommand{\bm}{{\mathbf{b}}}
\newcommand{\cm}{{\mathbf{c}}}
\newcommand{\dm}{{\mathbf{d}}}
\newcommand{\eem}{{\mathbf{e}}}
\newcommand{\fm}{{\mathbf{f}}}

\newcommand{\wm}{{\mathbf{w}}}
\newcommand{\p}{{\mathfrak{p}}}

\newcommand{\tu}{{\textunderscore}}

\begin{document}

\title[Polarization of neural codes] {Polarization of neural codes}

\author{Katie Christensen}
\address{Department of Mathematics\\ 
University of Louisville\\
Louisville, KY 40292, USA}
\email{katie.christensen@louisville.edu}
\thanks{$\dag$ the corresponding author}
\author{Hamid Kulosman$^\dag$}
\address{Department of Mathematics\\ 
University of Louisville\\
Louisville, KY 40292, USA}
\email{hamid.kulosman@louisville.edu}

\subjclass[2010]{Primary 13B25, 13F20, 13P25; Se\-con\-dary 92B20, 94B60}

\keywords{Neural code; Neural ideal; Canonical form; Minimal prime ideals; Motifs; Polarization; Pseudomonomial ideals; Square free monomial ideals}

\date{}

\begin{abstract} 
The neural rings and ideals as an algebraic tool for analyzing the intrinsic structure of neural codes were introduced by C.~Curto et al. in 2013. Since then they were investigated in several papers, including the 2017 paper by G\"unt\"urk\"un et al., in which the notion of polarization of neural ideals was introduced. In this paper we extend their ideas by introducing the notions of polarization of motifs and neural codes. We show that the notions that we introduced have very nice properties which could allow the studying of the intrinsic structure of neural codes of length $n$ via the square free monomial ideals in $2n$ variables and interpreting the results back in the original neural code ambient space. 

In the last section of the paper we introduce the notions of inactive neurons, partial neural codes, and partial motifs, as well as the notions of polarization of these codes and motifs. We use these notions to give a new proof of a theorem from the paper  by G\"unt\"urk\"un et al. that we mentioned above.
\end{abstract}

\maketitle

\section{Introduction}

The neural rings and ideals as an algebraic tool for analyzing the intrinsic structure of neural codes were introduced by C.~Curto et al. in 2013 in the pioneering paper \cite{civy}. In order to make our paper self-contained, we will give in this section all the definitions and facts from \cite{civy} that we are going to use, which are related to neural codes. All other notions and facts (that we assume are well-known) can be found either in \cite{civy}, or in the standard references \cite{am} and \cite{cls}.

\begin{definition_basic_facts}[{\cite{civy}}]
An element $\wm=w_1\dots w_n$ of $\F_2^n$ is called a {\it word} (of length $n$).
A set $\mathcal{C}\subseteq \F_2^n$ is called a {\it neural code}, shortly {\it code} (of length $n$). We also call the subsets of $\F_2^n$ {\it varietes} in $\F_2^n$. 
The code $\DC=\F_2^n\setminus \CC$ is called the {\it complement} of the code $\CC$ and is denoted by $^c\CC$.
We denote $\M=\{0,1,*\}$. We say that this set is the {\it set of motifs of length $1$}. We define a partial order on $\M$ by declaring that $0<*$ and $1<*$. A sequence $\am=a_1\dots a_n\in\M^n$ is called a {\it motif} ({\it of length $n$}).We define a partial order on the set $\M^n$ by declaring that $\am\le \bm$ if $a_i\le b_i$ for every $i\in [n]$.  In other words, $\am\le\bm$ if for each $i\in [n]$, $b_i=0$ (resp. $1$) implies $a_i=0$ (resp. $1$). We have
\[\am\le\bm \; \Leftrightarrow \; V_\am\subseteq V_\bm.\] 
For $\am \in\M^n$, the subset $V_\am$ of $\F_2^n$ consisiting of all the words $\wm$ obtained by replacing the stars of $\am$ by elements of $\F_2$  is called the {\it variety} of $\am$.

For a code $\mathcal{C}\subseteq \F_2^n$, a motif $\am$ of length $n$ is called a {\it motif of $\mathcal{C}$} if $V_\am \subseteq\mathcal{C}$. The set of all motifs of $\mathcal{C}$ is denoted by $\mathrm{Mot}(\mathcal{C})$. A motif $\am\in\mathrm{Mot}(\mathcal{C})$ is called a {\it maximal motif} of $\mathcal{C}$ if for any motif $\bm\in\mathrm{Mot}(\mathcal{C})$, $\am\le\bm$ implies $\am=\bm$. The set of all maximal motifs of $\mathcal{C}$ is denoted by $\mathrm{MaxMot}(\mathcal{C})$. For any $\am\in \mathrm{Mot}(\mathcal{C})$ there is a $\bm\in \mathrm{MaxMot}(\mathcal{C})$ such that $\am\le\bm$. We have $\CC=\emptyset$ if and only if $\mathrm{MaxMot}(\mathcal{C})=\emptyset$. Moreover, for any two codes $\CC^1$ and $\CC^2$, 
\[\CC^1=\CC^2 \Leftrightarrow \mathrm{MaxMot}(\CC^1)=\mathrm{MaxMot}(\CC^2).\]
\end{definition_basic_facts}

\begin{remark}[{\cite[pages 1593 and 1594]{civy}}]\label{MaxMot_subset}
We have
\begin{equation*}
\mathcal{C}=\cup\, \{V_\am\;:\;\am\in\mathrm{MaxMot}(\mathcal{C})\},
\end{equation*}
however it can happen that for a proper subset $M$ of $\mathrm{MaxMot}(\mathcal{C})$ we still have
\begin{equation*}
\mathcal{C}=\cup\, \{V_\am\;:\;\am\in M\}.
\end{equation*}
For example, consider the neural code $\mathcal{C}=\{000, 001, 011, 111\}\subseteq \F_2^3$. Then $\mathrm{MaxMot}(\mathcal{C})=\{00*, 0\!*\!1, *11\}$, however $\mathcal{C}=V_{00*}\cup V_{*11}$.
\end{remark}

\begin{definition}[{\cite{cls}}, {\cite{civy}}]
For a variety $V\subseteq \F_2^n$ we define the {\it ideal of $V$}, $\mathcal{I}(V)\subseteq \F_2[X_1, \dots, X_n]$,  in the following way: 
\begin{equation*}
\mathcal{I}(V)=\{f\in \F_2[X_1,\dots, X_n]\;:\: f(\wm)=0 \text{ for all } \wm\in V\}.
\end{equation*}

(Note that for any variety $V$ in $\F_2^n$ we have $\mathcal{I}(V)\supseteq \mathcal{B}$, where $\mathcal{B}=(X_1^2-X_1, \dots, X_n^2-X_n)$ is the {\it Boolean ideal} of $\F_2[X_1, \dots, X_n]$. Moreover, for $V\subseteq \F_2^n$ we have $\mathcal{I}(V)=\mathcal{B}$ if and only if $V=\F_2^n$.)

For an ideal $I\subseteq \F_2[X_1, \dots, X_n]$ we define the {\it variety of $I$}, $\mathcal{V}(I)\subseteq \F_2^n$, in the following way: 
\begin{equation*}
\mathcal{V}(I)=\{\wm\in\F_2^n\;:\: f(\wm)=0 \text{ for all } f\in I\}.
\end{equation*}
\end{definition}

\begin{theorem}[{\cite{civy}}, {\cite{g}}]
For every variety $V\subseteq \F_2^n$ we have
\begin{equation*}
\mathcal{V}(\mathcal{I}(V))=V.
\end{equation*}
For every ideal $I\subseteq \F_2[X_1, \dots, X_n]$ we have
\begin{equation*}
\mathcal{I}(\mathcal{V}(I))=\sqrt{I}=I+\mathcal{B}.
\end{equation*}
\end{theorem}

The second formula in the previous theorem is called the {\it Hilbert's Nullstellensatz for $\F_2$}.

\begin{definition}[{\cite{civy}}, {\cite{g}}]
For a motif $\am\in \M^n$ we define the {\it Lagrange polynomial} of $\am$, $L_\am\in \F_2[X_1,\dots, X_n]$, in the following way:
\begin{equation*}
L_\am=\prod_{a_i=1} X_i \prod_{a_j=0} (1-X_j).
\end{equation*}
\end{definition}

Note that for any word $\wm\in \F_2^n$, $L_\am(\wm)=1$ if and only if $\wm\in V_\am$ (i.e., $L_\am(\wm)=0$ if and only if $\wm\notin V_\am$).

\begin{definition}[{\cite[page 1582]{civy}}]
For a neural code $\CC\subseteq\F_2^n$ we define the {\it neural ideal} of $\CC$, $J_\CC\subseteq \F_2[X_1,\dots, X_n]$, in the following way:
\begin{equation*}
J_\CC=(\{L_\wm\;:\;\wm\in \,^c\CC\}).
\end{equation*}
\end{definition}

\begin{proposition}[{\cite[Lemma 3.2]{civy}}]
For a neural code $\CC\subseteq \F_2^n$ we have:
\begin{align*}
\mathcal{V}(J_\CC) & = \CC,\\
\mathcal{I}(\CC) & = J_\CC+\mathcal{B}.
\end{align*}
\end{proposition}

\begin{definition}[{\cite[page 1585]{civy}}]
A polynomial $f\in \F_2[X_1,\dots, X_n]$ is called a {\it pseudo-monomial} if it has the form
\begin{equation*}
f=\prod_{i\in\sigma} X_i\, \prod_{j\in\tau} (1-X_j)
\end{equation*}
for some $\sigma, \tau\subseteq [n]=\{1,\dots, n\}$ with $\sigma\cap \tau=\emptyset$.

An ideal $I\subseteq \F_2[X_1,\dots, X_n]$ is called a {\it pseudo-monomial ideal} if $I$ can be generated by a finite set of pseudo-monomials. 
\end{definition}

\begin{definition}[{\cite[page 1585]{civy}}]
Let $I$ be an ideal in $\F_2[X_1,\dots, X_n]$ and $f\in I$ a pseudo-monomial. We say that $f$ is a {\it minimal pseudo-monomial} of $I$ if there does not exist another pseudo-monomial $g\in I$ such that $\deg(g)<\deg(f)$ and $g\mid f$ in $F_2[X_1,\dots, X_n]$.
\end{definition}

\begin{definition}[{\cite[page 1585]{civy}}]
Let $I$ be a pseudo-monomial ideal in $F_2[X_1,\dots, X_n]$. We call the (finite) set $CF(I)$, consisting of all minimal pseudo-monomials of $I$, the {\it canonical form} of $I$.
\end{definition}

\begin{remark}[{\cite[page 1585]{civy}}]
Clearly, for any pseudo-monomial ideal $I$ of $F_2[X_1,\dots, X_n]$ , $CF(I)$ is unique and $I=(CF(I))$. On the other hand, $CF(I)$ is not necessarily a minimal generating set of $I$. For example, consider the ideal $I=(X_1(1-X_2), X_2(1-X_3))$. This ideal contains a third minimal pseudo-monomial: $X_1(1-X_3)=(1-X_3)\cdot [X_1(1-X_2)] + X_1\cdot [X_2(1-X_3)]$, so that $CF(I)=\{(X_1(1-X_2), X_2(1-X_3), X_1(1-X_3)\}$, which is not a minimal generating set of $I$.
\end{remark}

\begin{proposition}[{\cite[Proposition 4.5]{civy}}]\label{max_mot_complement}
Let $\CC\subseteq \F_2^n$ be a code in $\F_2^n$ and $\DC$ its complement. Let
\[\mathrm{MaxMot}(\CC)=\{\am^1, \dots, \am^l\}.\]
Then 
\begin{align*}
\mathrm{MaxMot}(\DC)=\,&\{\bm=b_1\dots b_n\;:\;\\
                                 &[\,(\forall b_i\ne \ast)(\exists \am^j)\; b_i=\overline{\am^j_i}\,]\;\; \mathrm{and}\\
                                 &[\,(\forall \am^j\ne \ast\dots\ast)(\exists b_i\ne\ast)\; b_i=\overline{\am^j_i}\,]\;\; \mathrm{and}\\
                                 &[\,\text{$\bm$ is maximal with respect to these two properties}\,]\}.
\end{align*}
In particular, if $\mathrm{MaxMot}(\CC)=\{\ast\dots\ast\}$, then $\mathrm{MaxMot}(\DC)=\emptyset$.
\end{proposition}

\begin{proof}
The proposition follows from Proposition 4.5 and Corollary 5.5 from \cite{civy}.
\end{proof}

\begin{proposition}[{\cite[Lemma 5.7]{civy}}]\label{CF_max_mot_compl}
Let $\CC\subseteq\F_2^n$ be a neural code and $J_\CC$ the neural ideal of $\CC$. Then
\begin{equation*}
CF(J_\CC)=\{L_\am\;:\;\am\in\mathrm{MaxMot}(^c\CC)\}.
\end{equation*}
\end{proposition}

\begin{remark}[{\cite[page 1594]{civy}}]
Note that it can happen that $J_\CC=(\{L_\am\;:\;\am\in M\})$, where $M$ is a proper subset of $\mathrm{MaxMot}(^c\CC)$. For example, for the neural code $\mathcal{C}=\{000, 001, 011, 111\}\subseteq \F_2^3$, we have $^c\CC=\{100, 010, 110, 101\}$, so that $\mathrm{MaxMot}(^c\CC)=\{10*, *10, 1\negthinspace*0\}$. Hence $CF(J_\CC)=\{L_{10*}, L_{*10}, L_{1*0}\}$. However, $J_\CC=(L_{10*}, L_{*10})$.
\end{remark}

\begin{definition}[{\cite[page 1594]{civy}}]
For a motif $\am\in \{0,1,*\}^n$ we define a {\it prime ideal} of $\am$, $\p_\am\subseteq \F_2[X_1,\dots, X_n]$, in the following way:
\begin{equation*}
\p_\am=(\{X_i\;:\;a_i=0\}\cup \{1-X_j\;:\;a_j=1\}).
\end{equation*}
If a prime ideal $\p$ in $\F_2[X_1, \dots, X_n]$ is equal to $\p_\am$ for some motif $\am$, we say that $\p$ is a {\it motivic prime ideal}.
\end{definition}

\begin{proposition}[{\cite[page 1594]{civy}}]
Let $\am, \bm\in \{0,1,*\}^n$ be two motifs of length $n$. We have:
\begin{align*}
V_\am \subseteq V_\bm & \Leftrightarrow \p_\bm\subseteq \p_\am,\\
\mathcal{I}(V_\am) & \,= \p_\am+\mathcal{B},\\
\mathcal{V}(\p_\am) & \,= V_\am.
\end{align*} 
\end{proposition}

For an ideal $I$ in $\F_2[X_1,\dots, X_n]$ we denote by $\mathrm{Min}(I)$ the set of all minimal prime ideals of $I$.

\begin{proposition}[{\cite[Lemma 5.1, Lemma 5.3 and Corollary 5.5]{civy}}]\label{Min_max_motifs}
Let $\CC\subseteq \F_2^n$ be a neural code and $\am\in\M^n$ a motif. We have:
\begin{align*}
\am \in\mathrm{Mot}(\CC) & \Leftrightarrow \p_\am\supseteq J_\CC,\\
\am \in\mathrm{MaxMot}(\CC) & \Leftrightarrow \p_\am\in\mathrm{Min}(J_\CC).
\end{align*}
Moreover,
\begin{equation}\label{Min_J}
\mathrm{Min}(J_\CC)=\{\p_\am\;:\;\am \in\mathrm{MaxMot}(\CC)\}.
\end{equation}
\end{proposition}

\begin{proposition}[{\cite[Corollary 5.5]{civy}}]
Let $\CC\subseteq \F_2^n$ be a non-empty neural code. Then
\begin{equation*}
J_\CC=\cap\,\{\p_\am\;:\;\am \in\mathrm{MaxMot}(\CC)\}
\end{equation*}
is the unique irredundant primary decomposition of $J_\CC$.
\end{proposition}

 The notions of polarization of pseudo-monomials and pseudo-monomial ideals were introduced in 2017 in the paper \cite{gjs} by G\"unt\"urk\"un et al.

\begin{definition}[{\cite[page 6]{gjs}}]\label{pol_pm}
For a pseudo-monomial 
\[f=\prod_{i\in\sigma}X_i\,\prod_{j\in\tau}(1-X_j)\in\F_2[X_1,\dots, X_n],\] 
where $\sigma,\tau$ are two disjoint subsets of $[n]$, we define its {\it polarization} $f^p$ to be the square-free monomial
\[f^p=\prod_{i\in\sigma}X_i\,\prod_{j\in\tau}Y_j\in\F_2[X_1,\dots, X_n, Y_1,\dots, Y_n].\]
\end{definition} 

\begin{proposition}[{\cite[Lemma 3.1]{gjs}}]\label{Gunt_lemma}
Let $f, g\in \F_2[X_1,\dots, X_n]$ be two pseudomonomials. Then
\[f\mid g \; \Leftrightarrow \; f^p\mid g^p.\] 
\end{proposition}

\begin{definition}[{\cite[Definition 3.3]{gjs}}]\label{pol_ideal}
Let $J$ be a pseudo-monomial ideal in $\F_2[X_1,\dots, X_n]$ and let $CF(J)=\{f_1, \dots, f_l\}$ be its canonical form. We define the {\it polarization of $J$} to be the ideal
\[J^p=(f_1^p, \dots, f_l^p) \subseteq \F_2[X_1, \dots, X_n, Y_1, \dots, Y_n].\]
\end{definition}

Note that then 
\[CF(J^p)=\{f_1^p, \dots, f_l^p\}.\]

\begin{remark}[about definitions and notation related to $\F_2^{2n}$ vs those related to $\F_2^n$]\label{def_n_2n}
Square-free monomial ideals are the ideals generated by square-free monomials and they are easier to deal with than the pseudo-monomial ideals. The previous two definitions show that, in order to get some conclusions about the pseudo-monomial ideals in $n$ variables $X_1,\dots, X_n$,  we can consider some related square-free monomial ideals in $2n$ variables, which, however, are not denoted by $X_1, \dots, X_{2n}$, but by $X_1,\dots, X_n, Y_1, \dots, Y_n$. Because of this difference in the notation for variables, we should be aware that, for example, a {\it pseudomonomial} in $\F_2[X_1,\dots, X_n, Y_1,\dots, Y_n]$ is a polynomial of the form
\[f=\prod_{i\in\sigma}X_i\,\prod_{j\in\tau}(1-X_j)\prod_{k\in\mu}Y_k\,\prod_{l\in\nu}(1-Y_l),\]
where $\sigma, \tau, \mu, \nu\subseteq [n]$,\, $\sigma\cap \tau=\emptyset$,\, $\mu\cap\nu=\emptyset$. Similarly, we have, for example, that for a motif $\am=b_1\dots b_n\,c_1\dots c_n\in \mathrm{Mot}(2n)$, the {\it Lagrange polynomial} $L_\am$ of $\am$ and the {\it prime ideal} $\p_\am$ of $\am$ are respectively given in the following way:
\begin{equation}\label{LP_2n}
L_\am=\prod_{a_i=1}X_i\,\prod_{a_j=0}(1-X_j)\prod_{b_i=1}Y_i\,\prod_{b_j=0}(1-Y_j),
\end{equation}
\begin{equation}\label{p_a_2n}
\p_\am=(\{X_i \,:\,b_i=0\} \cup \{1-X_j\,:\, b_j=1\} \cup \{Y_i \,:\,c_i=0\} \cup \{1-Y_j \,:\,c_j=1\}).
\end{equation}
So the definitions of these notions with respect to $\F_2^{2n}$ are the same as the ones with respect to $\F_2^n$, we just need to take into account the notation for the variables. This works for other notions as well (like, for example, {\it minimal pseudo-monomials} in an ideal, the {\it neural ideal} of a code, the {\it canonical form} of a pseudo-monomial ideal, etc.), while some notions (like, for example, {\it minimal primes }of an ideal, etc.) can be given in the form that does not depend on the notation for the variables. 

From now on we have the following convention:  if the length of motifs and codes is denoted by $n$, then the associated rings and ideals will always be in $n$ variables $X_1, \dots, X_n$, while for the length denoted by $2n$ the associated rings and ideals will always be in $2n$ variables $X_1,\dots, X_n, Y_1,\dots, Y_n$. For lengths given by concrete numbers it will always be clear from the context if the number is $n$ or $2n$.
\end{remark}

%------------------------------------
\section{Definitions of the polarizations of motifs and codes}\label{definitions}

We would like to define the motif $\am^p$ which is the polarization of a motif $\am=a_1\dots a_n\in\M^n$. Since for the motifs from $\M^n$ we have that the Lagrange polynomials and the prime ideals of motifs are in $n$ variables, and the polarizations of those Lagrange polynomials and prime ideals of motifs are in $2n$ variables, it is natural to try to define $\am^p$ to be an element of $\M^{2n}$. After defining the polarization of the motifs, we would define the polarization of a neural code $\CC\subseteq \F_2^n$ in the following way:
\begin{equation*}
\CC^p=\cup\,\{V_{\am^p}\;:\;\am\in\mathrm{MaxMot}(\CC)\}\subseteq \F_2^{2n}.
\end{equation*}

\noindent
This would imply that 
\begin{equation}\label{guiding_formula_1}
\mathrm{MaxMot}^p(\CC)=\mathrm{MaxMot}(\CC^p).
\end{equation} 
Here, and in the rest of the paper, we use the notation 
\begin{equation*}
M^p=\{\am^p\;:\;\am\in M\}
\end{equation*}
for any $M \subseteq \M^n$. We would also like that the formula 
\begin{equation}\label{guiding_formula_2}
\mathrm{Min}(J_{\CC^p})=\mathrm{Min}^p(J_\CC)
\end{equation}
holds. 
We will first try to find the formula for $\am^p$ in an example. So we will assume that the relations (\ref{guiding_formula_1}), and (\ref{guiding_formula_2}) hold. In addition to that, we will be using the relation (\ref{Min_J}):
\begin{equation*}
\mathrm{Min}(J_\CC)=\{\p_\am\;:\;\am \in\mathrm{MaxMot}(\CC)\}.
\end{equation*}

\begin{example}\label{example1}
Suppose that the relations (\ref{Min_J}), (\ref{guiding_formula_1}) and (\ref{guiding_formula_2}) hold. Let $\CC=\{10\}\subseteq \F_2^2$. Then
\begin{equation*}
\text{MaxMot}(\CC)=\{10\}.
\end{equation*}
By (\ref{guiding_formula_2}) we have 
\[\mathrm{MaxMot}(\CC^p)=\{10^p\}.\]
By (\ref{Min_J}),
\[\mathrm{Min}(J_\CC)=\{(1-X_1, X_2)\}=\p_{10},\]
hence by (\ref{guiding_formula_2}),
\[\mathrm{Min}(J_{\CC^p})=\{(1-X_1, X_2)^p\}=\{X_2, Y_1\}=\p_{*00*}.\]
Hence by (\ref{Min_J}),
\[\mathrm{MaxMot}(\CC^p)=\{*00*\}=\mathrm{MaxMot}^p(\CC)=\{10^p\}.\]
Thus
\begin{equation}\label{polar_formula_1}
10^p = *00\!*.
\end{equation}
Now $\DC =\,^c\CC=\{00, 01, 11\}$, so that
\[\mathrm{MaxMot}(\DC)=\{0*, *1\}.\]
Hence by (\ref{guiding_formula_1}),
\[\mathrm{MaxMot}(\DC^p)=\{0*^p, *1^p\}.\]
By (\ref{Min_J}),
\[\mathrm{Min}(J_\DC)=\{(X_1), (1-X_2)\},\]
hence by (\ref{guiding_formula_2}),
\[\mathrm{Min}(J_{\DC^p})=\{(X_1), (1-X_2)\}^p=\{(X_1), (Y_2)\}=\{\p_{0***}, \p_{***0}\}.\]
Hence by (\ref{Min_J}),
\[\mathrm{MaxMot}(\DC^p)=\{0\!*\!**, \, *\!*\!*0\}=\mathrm{MaxMot}^p(\DC)=\{0*^p, *1^p\}.\]
Hence we have either
\begin{equation}\label{pol_1}
0*^p=0\!*\!**, \quad *1^p=*\!*\!*0,
\end{equation}
or 
\begin{equation}\label{pol_2}
*1^p=*\!*\!*0, \quad 0*^p=0\!*\!**.
\end{equation}

\medskip\noindent
Now if we have
\begin{equation*}
(a_1a_2)^p=b_1b_2|c_1c_2,
\end{equation*}
then the formula (\ref{polar_formula_1}) would suggest that
\begin{align*}
a_1=1 &\Rightarrow b_1=*, \, c_1=0,\\
a_2=0 &\Rightarrow b_2=0, \, c_2=*,
\end{align*}
and the formula (\ref{pol_1}) would suggest that
\begin{align*}
a_1=0&\Rightarrow b_1=0, \,\, c_1=*,\\
a_2=* &\Rightarrow b_2=*, \,\, c_2=*.
\end{align*}

This would suggest that if
\[(a_1\dots a_n)^p=b_1\dots b_n|c_1\dots c_n,\]
then 
\begin{align*}
a_i=0 &\Rightarrow b_i=0, \,\,c_i=*,\\
a_i=1 &\Rightarrow b_i=*, \,\,c_i=0,\\
a_i=* &\Rightarrow b_i=*, \,\,c_i=*.
\end{align*}
(This, in particular, suggests that the formulas (\ref{pol_1}) are the correct way to polarize and not the formulas (\ref{pol_2}).)
\end{example}

\medskip
Having completed this example, based on the reasoning in it, we introduce the following definition.
\begin{definition}\label{pol_motif}
Let $\am=a_1\dots a_n\in\M^n$. We define its {\it polarization}
\begin{equation*}
\am^p=b_1\dots b_n|c_1\dots c_n
\end{equation*}
in the following way:
\begin{align*}
\text{if } a_i=0 \text{ in } \am,& \text{ then } b_i=0, \;c_i=* \text{ in } \am^p;\\
\text{if } a_i=1 \text{ in } \am,& \text{ then } b_i=*, \;c_i=0 \text{ in } \am^p;\\
\text{if } a_i=* \text{ in } \am,& \text{ then } b_i=*, \;c_i=* \text{ in } \am^p.
\end{align*}
\end{definition}

\medskip
Schematically:
\begin{align}\label{formula13}
\dots 0 \dots \quad &\mapsto \quad \dots 0 \dots | \dots * \dots\notag\\
\dots 1 \dots \quad &\mapsto \quad \dots * \dots | \dots 0 \dots\\
\dots * \dots \quad &\mapsto \quad \dots * \dots | \dots * \dots\notag
\end{align}

\medskip
We now define the polarization of a code.
\begin{definition}\label{pol_code}
For any code $\CC\subseteq \F_2^n$ we define its polarization $\CC^p\subseteq \F_2^{2n}$ in the following way:
\[\CC^p=\cup\{V_{\am^p}\;|\;\am\in\text{MaxMot}(\CC)\}.\]
\end{definition}

\begin{example}\label{example2}
Let's now determine $\CC^p$ and $\DC^p$ for $\CC$ and $\DC=^c\!\CC$ from Example \ref{example1}:
\begin{align*}
\CC^p&=\{10\}^p\\
           &=V_{(10)^p}\\
           &=V_{*00*}\\
           &=\{0000, 1000, 0001, 1001\},
\end{align*}
\begin{align*}
\DC^p=&\{00,01,11\}^p\\
           =&V_{(0*)^p} \cup V_{(*1)^p}\\
           =&V_{0***}\cup V_{***0}\\
           =&\{0000, 0100, 0010, 0110, 0001, 0101,\\
              &\;\; 0011, 0111, 1000, 1100, 1010, 1110\}.
\end{align*}

Note that here $\CC^p\cap \DC^p=\{0000\}$ and $\CC^p\cup \DC^p=\F_2^4\setminus\{1111\}$. In general, $\CC^p\cap \DC^p$ can contain several words and the complement of $\CC^p\cup \DC^p$ in $\F_2^{2n}$ can, as well, contain several words. 
\end{example}

\begin{definition}\label{dep_motif}
We say that a motif $\bm\in\M^{2n}$ is a {\it polar motif} if there is a motif $\am\in\M^n$ such that $\bm=\am^p$. The motif $\am$ is unique and we then denote $\am=\bm^d$.
\end{definition}

Note that we have
\begin{align*}
\am^{pd}=\am &\text{\quad for every $\am\in\M^n$},\\
\bm^{dp}=\bm &\text{\quad for every polar motif $\bm\in\M^{2n}$}.
\end{align*}

%-----------------------
\section{Properties of the polarization of motifs and codes}

\begin{proposition}\label{a_p_le_b_p}
Let $\am,\bm\in \M^n$. Then
\[\am\le\bm \; \Leftrightarrow \; \am^p\le\bm^p.\]
\end{proposition}
\begin{proof}
Suppose $\am\le\bm$. Let $i\in [n]$. If $(\bm^p)_i=0$, then $(\bm^p)_{n+i}=*$ and $b_i=0$, hence $a_i=0$, hence $(\am^p)_i=0$. If $(\bm^p)_{n+i}=0$, then $(\bm^p)_i=*$ and $b_i=1$, hence $a_i=1$, hence $(\am^p)_{n+i}=0$. Thus $\am^p\le\bm^p$.

Suppose $\am^p\le\bm^p$. Let $i\in [n]$. If $b_i=0$, then $(\bm^p)_i=0$, hence $(\am^p)_i=0$, hence $a_i=0$. If $b_i=1$, then $(\bm^p)_{n+i}=0$, hence $(\am^p)_{n+i}=0$, hence $a_i=1$. Thus $\am\le\bm$.   
\end{proof}

\begin{corollary}\label{mot_p_C_subset_mot_C_p}
For any code $\CC\subseteq \F_2^n$ we have
\[\mathrm{Mot}^p(\CC) \; \subseteq \; \mathrm{Mot}(\CC^p).\]
\end{corollary}
\begin{proof}
Let $\am\in\mathrm{Mot}(\CC)$ and let $\bm\in\mathrm{MaxMot}(\CC)$ such that $\am\le\bm$. By Proposition \ref{a_p_le_b_p}, $\am^p\le\bm^p$. Since (by the definition of $\CC^p$) $\bm^p\in\mathrm{Mot}(\CC^p)$, we have $\am^p\in\mathrm{Mot}(\CC^p)$.
\end{proof}

\begin{theorem}\label{maximal_motifs}
For any code $\CC\subseteq \F_2^n$ we have
\[\mathrm{MaxMot}(\CC^p) =\mathrm{MaxMot}^p(\CC).\]
\end{theorem}
\begin{proof}
Claim 1. $\mathrm{MaxMot}(\CC^p)\subseteq \{0,*\}^{2n}$.

\noindent
Proof of Claim 1. Suppose to the contrary, i.e., that $\bm\in\mathrm{MaxMot}(\CC^p)$ has a component $b_\alpha=1$ for some $\alpha\in [2n]$. Let  $\wm\in V_\bm$. Then $w_\alpha=1$. We have $\wm\in V_{\am^p}$ for some $\am\in\mathrm{MaxMot}(\CC)$. Then $(\am^p)_\alpha=*$, hence the word $\wm'$ obtained by replacing $w_\alpha$ in $\wm$ by $0$ is also in $V_{\am^p}$, hence in $\CC$. Hence the motif $\bm'$ obtained by replacing $b_\alpha$ by $*$ is also a motif of $\CC^p$, contradicting to the maximality of $\bm$. Claim 1 is proved.

\smallskip
Claim 2. Let $\bm\in\mathrm{MaxMot}(\CC^p)$. Then there is no $i\in [n]$ such that $b_i=b_{n+i}=0$.

\noindent
Proof of Claim 2. Suppose to the contrary. Let 
\begin{align*}
A &= \{j\in [n] \;:\; b_j=b_{n+j}=*\},\\
B &= \{j\in [n] \;:\; b_j=0, \, b_{n+j}=*\},\\
C &= \{j\in [n] \;:\; b_j=*, \, b_{n+j}=0\},\\
D &= \{j\in [n] \;:\; b_j=b_{n+j}=0\}.
\end{align*}
Then the sets $A,B,C,D$ form a partition of $[n]$ and $i\in D$. Let $\wm\in V_\bm$ be defined in the following way:
\begin{align*}
(\forall\, j\in A) \; & w_j=w_{n+j}=1;\\
(\forall\, j\in B) \; & w_j=0, \, w_{n+j}=1;\\
(\forall\, j\in C) \; & w_j=1, \, w_{n+j}=0;\\
(\forall\, j\in D) \; & w_j=w_{n+j}=0.
\end{align*}
Since $\wm\in \CC^p$, there is an $\am\in\mathrm{MaxMot}(\CC)$ such that $\wm\in V_{\am^p}$. Since $\am^p$ is a polar  motif, we have: 
\begin{align*}
(\forall\, j\in A) \; & (\am^p)_j=(\am^p)_{n+j}=*;\\
(\forall\, j\in B) \; & (\am^p)_j=0 \text{ or } *, \, (\am^p)_{n+j}=*;\\
(\forall\, j\in C) \; & (\am^p)_j=*, \, (\am^p)_{n+j}=0 \text{ or } *;\\
(\forall\, j\in D) \; & \text{at least one of $(\am^p)_j, \, (\am^p)_{n+j}$ is $*$, the other one is $0$ or $*$}.
\end{align*}
Since $D$ contains at least one element, these relations imply $\am^p>\bm$, contradicting to the maximality of $\bm$. Claim 2 is proved.

\smallskip
Claim 3. $\mathrm{MaxMot}(\CC^p)\subseteq \mathrm{MaxMot}^p(\CC)$.

\noindent
Proof of Claim 3. Let $\bm\in \mathrm{MaxMot}(\CC^p)$. By the claims 1 and 2, for each $i\in [n]$ we have one the following three cases: $b_i=b_{n+i}=*$, or, $b_i=0$, $b_{n+i}=*$, or, $b_i=*$, $b_{n+i}=0$. Let $\wm\in V_\bm$ be  a word defined in the following way: if $b_i=b_{n+i}=*$, then $w_i=w_{n+i}=1$; if $b_i=0$ and $b_{n+i}=*$, then $w_i=0$, $w_{n+i}=1$; if $b_i=*$ and $b_{n+i}=0$, then $w_i=1$, $w_{n+i}=0$. This word belongs to some $V_{\am^p}$, where $\am\in\mathrm{MaxMot}(\CC)$. Since $\am^p$ is a polar motif, we have the following cases: when $w_i=w_{n+i}=1$, then $(\am^p)_i=(\am^p)_{n+i}=*$; when $w_i=0$ and $w_{n+i}=1$, then $(\am^p)_i=0$, $(\am^p)_{n+i}=*$; when $w_i=1$ and $w_{n+i}=0$, then $(\am^p)_i=*$, $(\am^p)_{n+i}=0$. Hence $\am^p\ge \bm$. Since $\am^p\in\mathrm{Mot}(\CC)$ and $\bm\in\mathrm{MaxMot}(\CC)$, we have $\bm=\am^p$. Claim 3 is proved.

\smallskip
Claim 4. $\mathrm{MaxMot}^p(\CC)\subseteq \mathrm{MaxMot}(\CC^p)$.

\noindent
Proof of Claim 4. Suppose to the contrary. Let $\am\in  \mathrm{MaxMot}(\CC)$ such that $\am^p\notin  \mathrm{MaxMot}(\CC^p)$. By the definition of $\CC^p$, $\am^p\in\mathrm{Mot}(\CC^p)$, hence there is a $\bm\in\mathrm{MaxMot}(\CC^p)$ such that $\am^p<\bm$. By Claim 3, $\bm=\cm^p$ for some $\cm\in\mathrm{MaxMot}(\CC)$. Now by Proposition \ref{a_p_le_b_p}, from $\am^p<\cm^p$ we get $\am<\cm$, which is a contradiction since both $\am$ and $\cm$ are maximal motifs of $\CC$. Claim 4 is proved.

\smallskip
Now the statement of the theorem follows from Claim 3 and Claim 4.
\end{proof}

\begin{definition}\label{bar}
For a motif $\am=a_1\dots a_n\in\M^n$ we define $\overline{\am}$ to be the motif $\bm=b_1\dots b_{n}\in\M^n$ which satisfies the following condition\;
\[\text{for } i=1,2\dots, n, \text{ if } a_i\ne *, \text{ then } b_i=\overline{a_i}=1-a_i.\]
\end{definition}

\begin{example}
$\overline{1*01}\,=\,0*10$.
\end{example}

\medskip
Note that for two motifs $\am, \bm\in\M^n$ we have
\begin{equation}\label{p_d}
\bm=\overline{\overline{\am}^p}\; \Leftrightarrow \;\am=\overline{\overline{\bm}^d}.
\end{equation}

Also, if for any code $\CC\subseteq\F_2^n$ and any $M\subseteq \mathrm{Mot}(\CC)$ we denote $\overline{M}=\{\overline{\am}\;:\;\am\in M\}$, then
\begin{align}
\mathrm{Mot}(\overline{\CC})&=\overline{\mathrm{Mot}(\CC)},\label{mot_bar}\\
\mathrm{MaxMot}(\overline{\CC})&=\overline{\mathrm{MaxMot}(\CC)}.\label{max_mot_bar}
\end{align}

\begin{proposition}\label{pol_of_LP}
For any motif $\am\in\M^n$ we have
\[L_\am^p=L_{\overline{\overline{\am}^{\,p}}}.\]
\end{proposition}

\begin{proof}
This follows from the definitions (\ref{pol_pm}), (\ref{LP_2n}), (\ref{pol_motif}), and (\ref{bar}).
\end{proof}

\begin{example}
Let $n=4$ and let $\am=11*0\in\text{Mot}(4)$. Then
\[L_{\am}=X_1X_2(1-X_4).\]
Hence
\[L_{\am}^p=X_1X_2Y_4.\]
On the other side we have:
\begin{align*}
\overline{\overline{\am}^p}&=\overline{\overline{11*0}^p}\\
                                              &=\overline{00*11^p}\\
                                              &=\overline{00*0|***0}\\
                                              &=11**|***1.
\end{align*}
Hence 
\[L_{\overline{\overline{\am}^p}}=X_1X_2Y_4=L_{\am}^p.\]
\end{example}

\begin{definition}
We say that two motifs $\am, \bm\in\M^n$ are {\it disjoint} if there is an $i\in [n]$ such that $a_i=\overline{b_i}$. 
\end{definition}

\begin{definition}
On the set $\M$ of motifs of length $1$ we introduce a commutative operation of {\it addition} in the following way:
\begin{align*}
0+0&=0,\\
0+1&=1,\\
1+1&=0,\\
0+*&=*,\\
1+*&=*,\\
*+*&=*.
\end{align*}
The first three lines represent the arithmetic in the field $\F_2$, while the last three lines represent the max-arithmetic. 
We then the addition in $\M^n$ by adding two motifs componentwise.
\end{definition}

It is easy to verify that with this operation and the partial order that we introduced before, $\M^n$ is a {\it partially ordered monoid}. 

\medskip
The importance of above definition lies in the fact that the sum $\am+\bm$ of two motifs $\am, \bm\in\M^n$ has an $1$-component if and only if the motifs $\am$ and $\bm$ are disjoint. Thus {\it we can recognize the disjointness of two motifs algebraically} by considering their sum.

\medskip
\begin{proposition}\label{disjoint_complement}
Let $\am\in\mathrm{Mot}(\CC)$ for some code $\CC\subseteq \F_2^n$ and $\bm\in\M^n$. Then $\bm\in\mathrm{Mot}(^c\CC)$ if and only if $\bm$ is disjoint with $\am$. Moreover, the maximal motifs of $^c\CC$ are the motifs $\bm$ that are maximal among the motifs from $\M^n$ that are disjoint with all the maximal motifs of $\CC$. 
\end{proposition}
\begin{proof}
Easy to see.
\end{proof}

\begin{proposition}\label{component_one}
Let $\am, \bm, \bm' \in \M^n$. If $\am+\bm$ has an $1$-component and $\bm'\le\bm$, then $\am+\bm'$ has an $1$-component. 

In particular, if $\CC$ is a code in $\M^n$, the maximal motifs of $^c\CC$ are the maximal elements $\bm\in\M^n$ such that each $\am+\bm$ ($\am\in\mathrm{MaxMot}(\CC)$) has an $1$-component.  
\end{proposition}
\begin{proof}
Easy to see.
\end{proof}

\begin{corollary}
Let $\CC\subseteq \F_2^n$ be a code in $\F_2^n$. If $\textbf{b}\in\mathrm{MaxMot}(^c(\CC^p))$, then every $b_i$ which is different than $\ast$ is equal to $1$.
\end{corollary}

\begin{proof}
The statement follows from the previous proposition as each $0$ could be replaced by a $*$, which would result in a strictly bigger motif disjoint with all maximal motifs of $\CC^p$.
\end{proof}

\begin{proposition}\label{disjoint_motifs}
The motifs $\am$ and  $\bm$ from $\M^n$ are disjoint if and only if the motifs $\am^p$ and $\overline{\overline{\bm}^p}$ from $\M^{2n}$ are disjoint.
\end{proposition}
\begin{proof}
$\Rightarrow)$ Suppose that $\am$ and $\bm$ are disjoint. We first consider the case $a_i=1$, $b_i=0$ for some $i\in [n]$. Then $(\am^p)_i=*$ and $(\am^p)_{n+i}=0$, while $(\overline{\overline{\bm}^p})_i=*$ and  $(\overline{\overline{\bm}^p})_{n+i}=1$. Hence $\am^p$ and $\overline{\overline{\bm}^p}$ are disjoint. The case  $a_i=0$, $b_i=1$ for some $i\in [n]$ is similar.

$\Leftarrow)$ Suppose that $\am^p$ and $\overline{\overline{\bm}^p}$ are disjoint. We first consider the case $(\am^p)_i=0$,  $(\overline{\overline{\bm}^p})_i=1$ for some $i\in [n]$. Then $a_i=0$ and $(\overline{\bm}^p)_i=0$, hence $(\overline{\bm})_i=0$. Hence $b_i=1$, so that $\am$ and  $\bm$ are disjoint. The case  $(\am^p)_i=1$,  $(\overline{\overline{\bm}^p})_i=0$ for some $i\in [n]$ is similar.
\end{proof}

\begin{proposition}\label{D_subset_complement_C}
Let $\CC, \DC$ be two codes in $\F_2^n$. Then:
\[\DC\,\subseteq \, ^c\CC\; \Leftrightarrow\; \overline{\overline{\DC}^p}\, \subseteq\, ^c(\CC^p).\]
\end{proposition}
\begin{proof}
The next equivalences follow from Proposition \ref{disjoint_complement}, Proposition \ref{disjoint_motifs}, Theorem \ref{maximal_motifs}, and Proposition \ref{disjoint_complement}, respectively.
\begin{align*}
\DC\subseteq\, ^c\CC & \Leftrightarrow (\forall\,\am\in\mathrm{MaxMot}(\CC))(\forall\,\bm\in\mathrm{MaxMot}(\DC)) \text{ $\am$ and $\bm$ are disjoint}\\
                                   & \Leftrightarrow  (\forall\,\am\in\mathrm{MaxMot}(\CC))(\forall\,\bm\in\mathrm{MaxMot}(\DC)) \text{ $\am^p$ and $\overline{\overline{\bm}^p}$ are disjoint}\\                                   
                                    & \Leftrightarrow  (\forall\,\cm\in\mathrm{MaxMot}(\CC^p))(\forall\,\dm\in\mathrm{MaxMot}(\overline{\overline{\DC}^p})) \text{ $\cm$ and $\dm$ are disjoint}\\
                                    & \Leftrightarrow \overline{\overline{\DC}^p}\, \subseteq\, ^c(\CC^p).
\end{align*}
\end{proof}

\begin{remark}
Note that in the previous proposition the equality on the left hand side is not equivalent with the equality on the right hand side, as we are going to see in Example \ref{main_example}.
\end{remark}

\begin{corollary}\label{complement_C_corrolary}
Let $\DC=\, ^c\CC$. Then
\[\CC^p\, \subseteq\, ^{^{^c}}\overline{\overline{\DC}^p}.\]
\end{corollary}
\begin{proof}
Follows immediately from the previous proposition.
\end{proof}

The reason for giving the next definition and using the termino\-lo\-gy introduced in it will become clear later, after Theorem \ref{CF_main} and Example \ref{main_example}.

\begin{definition}
Let $\CC$ be a code in $\F_2^n$ and let $\DC$ be its complement. We call the code $\CC^{[p]}$, defined by
\begin{equation}
\CC^{[p]}= \,^{^{^c}}\overline{\overline{\DC}^p},\label{C_fp_formula}
\end{equation}
the {\it formal polarization} of the code $\CC$.
\end{definition}

\begin{proposition}\label{max_motifs_pol_formal_pol}
Let $\CC$ be a code in $\F_2^n$ and $\DC$ its complement. We have:
\begin{align}
\mathrm{MaxMot}(\CC^p) & \subseteq \mathrm{MaxMot}( \CC^ {[p]}  ),\label{mm_Cp_Cfp_1}\\
\mathrm{MaxMot}(^c(\CC^{[p]}))&=\overline{\overline{\mathrm{MaxMot}(\DC)} ^{\,p}}
\subseteq  \mathrm{MaxMot}(^c(\CC^p)).\label{mm_Cp_Cfp_2}
\end{align}
\end{proposition}
\begin{proof}
By Theorem \ref{maximal_motifs} and the formula (\ref{C_fp_formula}), this is equivalent with showing that 
\[ (\forall\,\am\in\mathrm{MaxMot}(\CC))(\forall\,\bm\in\mathrm{MaxMot}(\DC)) \text{ $\am^p$ and $\overline{\overline{\bm}^p}$ are disjoint},\]
which is true by Proposition \ref{disjoint_motifs}.

We now show (\ref{mm_Cp_Cfp_2}).
Let $\dm\in \mathrm{MaxMot}(^c(\CC^{[p]}))=\mathrm{MaxMot}(\overline{\overline{\DC}^p})$. Then $\dm=\overline{   \overline{\bm}^p   }$ for some $\bm \in\mathrm{MaxMot}(\DC)$. Let $\eem\in\mathrm{MaxMot}(^c(\CC^p))$ such that $\dm\le\eem$, i.e. $\overline{   \overline{\bm}^p   }\le \eem$. Hence $\overline{\bm}^p\le \overline{\eem}$. Since $\eem$ is bigger than or equal to a polar motif, $\eem$ is a polar motif too, so $\eem=\overline{\fm}^p$ for some motif $\fm$. Hence $\overline{\bm}^p\le\overline{\fm}^p$, so that 
 $\overline{\overline{\bm}^p}\le\overline{\overline{\fm}^p}=\eem$. Hence, since $\eem$ is disjoint with all the maximal  motifs of $\CC^p$, then by Proposition \ref{disjoint_motifs}, $\fm$ is disjoint with all the maximal motifs of $\CC$ and $\fm\ge\bm$, where $\bm$ is one of the maximal motifs among the motifs that are disjoint with all maximal motifs of $\CC$. Hence $\fm=\bm$, so that $\dm=\eem$.
\end{proof}

\begin{theorem}\label{CF_main} 
Let $\CC$ be a code in $\F_2^n$. We have:
\[CF(J_\CC^p) = CF^p(J_\CC) = CF(J_{\CC^{[p]}})
\subseteq CF(J_{\CC^p}).\]
\end{theorem}
\begin{proof}
Let $CF(J_\CC)=\{f_1,\dots, f_k\}$. By definition, $J_\CC^p=(f_1^p,\dots, f_k^p)$. Here $f_1^p$, \dots, $f_k^p$ are square-free monomials. By \cite[Corrolary 1.10]{he}, the set $\{f_1^p,\dots, f_k^p\}$ contains a minimal susbest $S$ (with respect to inclusion) which generates $J_\CC^p$. By \cite[Corrolary 1.8]{he}, if $f_i^p\notin S$, then $f_i^p\mid f_j^p$ for some $f_j^p\in S$. Then by Proposition \ref{Gunt_lemma}, $f_i\mid f_j$, a contradiction. Thus $S=\{f_1^p, \dots, f_k^p\}$. Hence by \cite[Proposition 1.11]{he}, $CF(J_\CC^p)=\{f_1^p,\dots, f_k^p\}=CF^p(J_\CC)$. 

\smallskip
Let $\DC$ be the complement of $\CC$. We have:
\begin{align*}
 CF^p(J_\CC)&=\{L_\am^p\;:\;\am\in\text{MaxMot}(\DC)\} \hspace{2.1cm} \text{(by Proposition \ref{CF_max_mot_compl})}\\
                      &=\{L_{\overline{\overline{\am}^p}}\;:\;\am\in\text{MaxMot}(\DC)\} \hspace{2cm} \text{(by Proposition \ref{pol_of_LP})}\\
                      &=\{ L_\bm\;:\;\overline{\overline{\bm}^d}\in\text{MaxMot}(\DC) \} \hspace{1.9cm} \text{(by (\ref{p_d}))}\\
                      &=\{L_\bm\;:\;\overline{\bm}^d\in\text{MaxMot}(\overline{\DC})\}\\
                      &=\{L_\bm\;:\;\overline{\bm}\in\text{MaxMot}(\overline{\DC}^p)\}\\
                      &=\{L_\bm\;:\;\bm\in\text{MaxMot}(\overline{\overline{\DC}^p})\}\\
                      &=\,CF(J_{\CC^{[p]}}). \hspace{4.65cm}\text{(by Proposition \ref{CF_max_mot_compl})}
\end{align*}

Finally, the inclusion in the statement of the theorem follows from Proposition \ref{max_motifs_pol_formal_pol} and Proposition
 \ref{CF_max_mot_compl}.
\end{proof}

\begin{example}\label{main_example}
Consider the neural codes
\[\CC=\{000, 100, 110, 011\} \text{ and } \DC=^c\CC=\{001, 010, 101, 111\}\]
in $\F_2^3$. We have
\[\mathrm{MaxMot}(\CC)=\{*00, 1\!*\!0, 011\} \text{ and } \mathrm{MaxMot}(DC)=\{*01, 1\!*\!1, 111\}.\]
Then by Theorem \ref{maximal_motifs},
\begin{align}
\mathrm{MaxMot}(\CC^p) &=\{*00\!*\!**,\, *\!*\!00\!*\!*,\, 0\!*\!*\!*\!00\},\label{main_ex_1}\\
\mathrm{MaxMot}(\overline{\overline{\DC}^p}) &=\{*\!*\!1\!*\!1*, \,1\!*\!1\!*\!**,\, *1\!*\!1\!*\!1\}\label{main_ex_2}.
\end{align}
By Proposition \ref{D_subset_complement_C} we have $\overline{\overline{\DC}^p}\, \subseteq\, ^c(\CC^p)$. From (\ref{main_ex_2}) we have by Proposition \ref{CF_max_mot_compl},
\[CF(J_{\CC^{[p]}})=\{X_3Y_2,\, X_1X_3,\, X_2Y_1Y_3\}=CF^p(J_\CC).\] 
By Proposition \ref{disjoint_complement}, a motif $\bm\in\M^6$ is a maximal motif of $^c(\CC^p)$ if and only if it is a maximal motif from $\M^6$ disjoint with all the maximal motifs of $\CC^p$. The sets 
\[A_1=\{2,3\}, \;\; A_2=\{3,4\}, \;\; A_3=\{1,5,6\},\]
are the sets of coordinates which the maximal motifs $\am^1=*00\!*\!**,\, \am^2=*\!*\!00\!*\!*,\, \text{ and }\am^3=0\!*\!*\!*\!00$ of $\CC^p$ have zeros at, respectively. To get a set $B$ of coordinates where a motif $\bm\in\mathrm{MaxMot}(^c(\CC^p))$ has ones, we need to take one element from each of the sets $A_1, A_2, A_3$, and then, out of all sets $B$ obtained in that way ($2\times 2\times 3=12$ of them) select the minimal ones with respect to inclusion. In that way we get the sets $B_1=\{2,4,1\}, \; B_2=\{2,4,5\}, \; B_3=\{2,4,6\}, \; B_4=\{3,1\}, \; B_5=\{3,5\},$ and $B_6=\{3,6\}$. For each of these sets $B_i$ we get an element $\bm^i\in\mathrm{MaxMot}(^c(\CC^p))$ by putting ones at all the coordinates of $B_i$ and stars at all other coordinates. Thus
\begin{align*}
\mathrm{MaxMot}(^c(\CC^p)) & = \{*\!*\!1\!*\!1*, \;\; 1\!*\!1\!*\!**, \;\; *1\!*\!1\!*\!1, \;\; *\!*\!1\!*\!*1, \\
                                                   &  \quad\;\;\, *\!1\!*\!11*,\;\; 11\!*\!1\!*\!*\}.
\end{align*}
Hence by Proposition \ref{CF_max_mot_compl},
\begin{equation*}
CF(J_{\CC^p}) = \{X_3Y_2, \; X_1X_3, \; X_2Y_1Y_3, \; X_3Y_3, \; X_2Y_1Y_2, \;X_1X_2Y_1\}.
\end{equation*}
In particular,
\begin{equation*}
\CC^p \; \subset \; \CC^{[p]}.
\end{equation*}
(One can check that $\CC^p$ has $29$ words, while $\CC^{[p]}$ has $35$ words.)
\end{example}

\begin{example}\label{example}
Consider again the code $\CC=\{10\}\subseteq \F_2^n$ and its complement $\DC=\{00,01,11\}$ from Examples \ref{example1} and \ref{example2}. We have:
\begin{align*}
\mathrm{MaxMot}(\CC) &= \{10\},\\
\mathrm{MaxMot}(\DC) &= \{0*, \; *1\},\\
CF(J_\CC) &= \{1-X_1, \; X_2\},\\
CF(J_{\CC^{[p]}}) &= CF^p(J_\CC) = \{X_2, \; Y_1\},\\
\CC^p &= \{0000, 1000, 0001, 1001\},\\
\mathrm{MaxMot}(^c(\CC^p)) &= \{*1\!*\!*, \; *\!*\!1*\},\\
CF(J_{\CC^p}) &= \{X_2, \; Y_1\}.
\end{align*}
Thus in this example $\CC^p=\CC^{[p]}$.
\end{example}

\begin{definition}
The prime ideals $\p\subseteq\F_2[X_1,\dots, X_n, Y_1, \dots, Y_n]$ such that $\p=\p_{\am^p}$ for some 
$\am\in\M^n$ are called {\it polar motivic primes}. 
\end{definition}

For polar motivic primes we have the following formula:
\begin{equation}\label{formula20}
\p_{\am^p} = \p_\am^p.
\end{equation}
Indeed, if $\am=a_1\dots a_n\in\mathrm{MaxMot}(\CC)$ and $\am^p=b_1\dots b_nc_1\dots c_n$, then
\begin{align*}
\p_{\am^p} &= (\{X_i \;:\;b_i=0\} \cup \{Y_j \;:\; c_j=0\})\\
                     &=(\{X_i \;|\;a_i=0\} \cup \{Y_j \;:\; a_j=1\})\\
                     &=\p_\am^p.
\end{align*}

\begin{theorem}\label{minimal_primes}
For any code $\CC\subseteq \F_2^n$ we have:
\[\mathrm{Min}(J_{\CC^p})=\mathrm{Min}^p(J_\CC)\subseteq \mathrm{Min}(J_{\CC^{[p]}}).\]
\end{theorem}

\begin{proof}
We have 
\begin{align*}
\mathrm{Min}(J_{\CC^p}) &= \{\p_\mathbf{d} \;:\; \mathbf{d}\in\mathrm{MaxMot}(\CC^p)\} \hspace{1.36cm}\text{(by Proposition \ref{Min_max_motifs})}\\
                                           &= \{\p_{\am^p} \;:\; \am\in\textrm{MaxMot}(\CC)\} \hspace{1.43cm} \text{(by Theorem \ref{maximal_motifs})}\\
                                           &= \{\p_{\am^p} \;:\; \p_\am \in \mathrm{Min}(J_\CC)\} \hspace{1.93cm} \text{(by Proposition \ref{Min_max_motifs})}\\
                                           &= \{\p_\am^p \;:\; \p_\am \in \mathrm{Min}(J_\CC)\}.  \hspace{2cm} \text{(by the formula (\ref{formula20}))}
\end{align*}
Hence
\begin{equation*}
\mathrm{Min}(J_{\CC^p})=\mathrm{Min}^p(J_\CC). 
\end{equation*}
The inclusion part of the statement follows from (\ref{mm_Cp_Cfp_1}) and the relation (\ref{Min_J}) from Proposition \ref{Min_max_motifs}. 
\end{proof}

\begin{example}\label{main_example_cont}
We continue Example \ref{main_example}. By (\ref{main_ex_2}) we have
\[\mathrm{MaxMot}(^c(\CC^{[p]}))=\{*\!*\!1\!*\!1*, \,1\!*\!1\!*\!**,\, *1\!*\!1\!*\!1\}.\]
Now using the same technique as in Example \ref{main_example} (for finding $\mathrm{MaxMot}(^c(\CC^p))$ given $\mathrm{MaxMot}(\CC^p)$) we find here that 
\begin{align*}
\mathrm{MaxMot}(\CC^{[p]}) & = \{*00\!*\!**, \;\; *\!*\!00\!*\!*, \;\; *\!*\!0\!*\!*0, \;\; 00\!*\!*0*, \\
                                                  & \quad \;\;\; \;0\!*\!*00*, \;\; 0\!*\!*\!*\!00\}.
\end{align*} 
Hence by Proposition \ref{Min_max_motifs} we have:
\begin{align*}
\mathrm{Min}(J_{\CC^p}) & = \{(X_2, X_3), \; (X_3, Y_1), \; (X_1, Y_2, Y_3)\}=\mathrm{Min}^p(J_\CC),\\
\mathrm{Min}(J_{\CC^{[p]}}) & = \{(X_2, X_3), \; (X_3, Y_1), \; (X_1, Y_2, Y_3),\\
                                           & \hspace{2.6cm}(X_3, Y_3), \; (X_1, X_2, Y_2), \; (X_1, Y_1, Y_2)\}.
\end{align*}
The minimal prime ideals of $J_{\CC}^p$ (i.e., $J_{\CC^{[p]}}$) were also calculated in \cite[Example 5.4]{gjs} in a different way. 

Note that among the minimal primes of $J_{\CC^{[p]}}$ we have, in addition to all the minimal primes of $J_{\CC^p}$, three non-polar minimal primes, namely $\p_{**0**0}=(X_3, Y_3)$, $\p_{00**0*}=(X_1, X_2, Y_2)$,  and $p_{0**00*}=(X_1, Y_1, Y_2)$. A natural question to ask is the following one: if for an $\am\in \M^{2n}$ we have $\p_\am\in\mathrm{Min}(J_{\CC^{[p]}})$, how is then the motif $\am$ related to $\CC$? 
A statement related to this question is given in the next section in Theorem \ref{thm_gjs}, which is \cite[Theorem 5.1]{gjs}. We will give a different proof of this theorem.
\end{example}

\begin{theorem}\label{irredundant_primary_decomposition}
For any code $\CC\subseteq \F_2^n$, the ideal $J_{\CC^p}$ has the unique irredundant primary decomposition and it is obtained by polarizing the prime ideals from the unique irredundant primary decomposition of $J_\CC$.
\end{theorem}

\begin{proof}
By \cite[Corollary 5.5]{civy}, the ideals $J_\CC$ and $J_{\CC^p}$ have the unique irredundant primary decompositions
\begin{align*}
J_\CC&=\cap\{\p_\am\;:\;\am\in\mathrm{MaxMot}(\CC)\},\\
J_\CC&=\cap\{\p_\bm\;:\;\bm\in\mathrm{MaxMot}(\CC^p)\}.
\end{align*}
Hence the statement follows from Theorem \ref{maximal_motifs} and the formula (\ref{formula20}).
\end{proof}

%-----------------------
\section{Partial motifs}
\begin{definition}
We denote $\PM=\{0,1,*,\tu\}$. We say that this set is the {\it set of partial motifs of length $1$}. We define a partial order on $\PM$ by declaring that $0<*$ and $1<*$. Note that $\tu$ is comparable only with itself (the same holds for $0$ and $1$). We define a partial order on the set $\PM^n$ by declaring that $\am\le \bm$ if $a_i\le b_i$ for every $i\in [n]$. A {\it partial motif} ({\it of length $n$}) is an element of $\PM^n$. A {\it partial word} ({\it of length $n$}) is an element of $\PW^n=\{0,1,\tu\}^n$. The neurons $i\in [n]$ for which $w_i=\tu$ are said to be {\it inactive}. A {\it partial code} {\it (of length $n$)} is a subset of $\PW^n$. The {\it variety} of a partial motif $\am$ is the set of all partial words obtained by replacing all the stars in $\am$ by zeros and ones. It is denoted by $V_\am$. If $\CC\subseteq \PW^n$ is a a partial code, then $\am\in\PM^n$ is a {\it partial motif of $\CC$} if $V_\am\subseteq \CC$. The set of all partial motifs of a partial code $\CC$ is denoted by $\mathrm{ParMot}(\CC)$. The set of all maximal partial motifs of a partial code $\CC$ is denoted by $\mathrm{MaxParMot}(\CC)$.
\end{definition}

\begin{example}
We can think of the partial word $\wm=\tu 01\tu 00\tu 1$ as of a statement that the neurons $3$ and $8$ are firing, the neurons $2, 5$ and $6$ are not firing, and the neurons $1,4,$ and $7$ are inactive.
\end{example}

The set of all partial motifs $\am\in\PM^n$ (resp. partial words $\wm\in \PW^n$) such that $a_{i_1}=\dots =a_{i_k}=\tu$ (resp. $w_{i_1}=\dots =w_{i_k}=\tu$) and all the remaining neurons are active, is denoted by $\PM^n_{i_1,\dots, i_k}$  (resp. 
 $\PW^n_{i_1,\dots, i_k})$. It is naturally in a bijective correspondence with the set $\M^{n-k}$ (resp. $\F_2^{n-k}$).
If $\am$ (resp. $\wm$) is a motif (resp. word), then the partial motif (resp. partial word) obtained by replacing each $a_i$ (resp. $w_i$), $i=i_1,\dots, i_k$, by $\tu$ is called the {\it partial motif} (resp. {\it partial word}) {\it obtained by deactivating the neurons $i_1, \dots, i_k$} and is denoted by $a^\tu_{i_1,\dots, i_k}$ (resp. $w^\tu_{i_1,\dots, i_k}$).

If $\CC\subseteq\F_2^n$ is a code and $\{i_1,\dots, i_k\}\subseteq [n]$, then the code obtained by replacing each $w_{i_r}$ $(r=1,\dots, k)$ by $\tu$  in each word $\wm\in\CC$ is called the {\it partial code obtained by deactivating the neurons $i_1,\dots, i_k$} and is denoted by $\CC^{^\tu}_{i_1,\dots, i_k}$. The partial code $\CC^{^\tu}_{i_1,\dots, i_k}$ is naturally in a bijective correspondence with the code $\CC_{i_1,\dots, i_k}$ obtained by deleting the neurons $i_1,\dots, i_k$.  

\begin{proposition}
Let $\CC$ be a code in $\F_2^n$ and let $\wm\in\PW^n_{i_1,\dots, i_k}$. If $\wm$ is not an element of $\CC^{^\tu}_{i_1,\dots, i_k}$, then the motif $\am$ obtained from $\wm$ by replacing each $\tu$ by $*$ belongs to $\mathrm{Mot}(^c\CC)$.
\end{proposition}
\begin{proof}
Easy to see.
\end{proof}

\begin{definition}\label{pol_part_motifs}
If $\am\in\PM^n$, then we define its {\it polarization} $\am^p=b_1\dots b_nb_{n+1}\dots b_{2n}\in \PM^{2n}$ defined by:
\begin{align*}
&b_i=0,\; b_{n+1}=*, \quad  \text{when } a_i=0;\\
&b_i=*,\; b_{n+1}=0, \quad  \text{when } a_i=1;\\
&b_i=*,\; b_{n+1}=*, \quad  \text{when } a_i=*;\\
&b_i=\tu, \; \;b_{n+1}=\tu, \quad\, \text{when } a_i=\tu.
\end{align*}
A partial motif $\bm\in\PM^{2n}$ is called a {\it polar partial motif} if $\bm=\am^p$ for some partial motif $\am\in\PM^n$. Then $\am=\bm^d$ is called the {\it depolarization} of the polar partial motif $\bm$. 
\end{definition}

The next theorem is a slight reformulation of Theorem 5.1 from \cite{gjs}. We give a different proof.

\begin{theorem}[{\cite [Theorem 5.1]{gjs}}]\label{thm_gjs}
Let $\cm\in \M^{2n}$ and let $\am\in \M^{2n}$ be the motif obtained by replacing all ones in $\cm$ by stars.
 Let $\{i_1,\dots, i_k\}$ be the set of all elements $i$ of $[n]$ such that $a_i=a_{i+n}$. Then $\p_\cm\supseteq J_{\CC^{[p]}}$ if and only if 
$\am^{\tu\, d}_{i_1,\dots, i_k, i_1+n,\dots, i_k+n} \in \mathrm{ParMot}(\CC^{^{\tu}}_{i_1,\dots, i_k})$. 
%Moreover, $\p_\cm$ is a minimal prime of $J_{\CC^{[p]}}$ if and only if $\cm=\am$ and $\am^{\tu\, d}_{i_1,\dots, i_k, i_1+n,\dots, i_k+n} \in \mathrm{MaxParMot}(\CC^{^{\tu}}_{i_1,\dots, i_k})$. 
\end{theorem}
\begin{proof}
Let $\DC=\, ^c\CC$. Since
\[\mathrm{MaxMot}(^c(\CC^{[p]}))=\{\,\overline{\overline{\bm}^p}\;:\;\bm\in\mathrm{MaxMot}(\DC)\},\]
the motifs $\cm$ of $\CC^{[p]}$ are the motifs from $\M^{2n}$ that are disjoint with all $\overline{\overline{\bm}^p}$, $\bm\in\mathrm{MaxMot}(\DC)$. They give the primes $\p_\cm$, and these are all the motivic primes that contain $J_{\CC^{[p]}}$. %and the minimal primes of $J_{\CC^{[p]}}$ are coming from maximal such motifs $\cm$. 
Note, however, that $\p_\cm\supset J_{\CC^{[p]}}$ if and only if $\p_\am\supset J_{\CC^{[p]}}$.
%so the minimal primes of $J_{\CC^{[p]}}$ are the primes $\p_\cm$ such that $\cm=\am$ and $\am$ is maximal among the motifs disjoint with all $\overline{\overline{\bm}^p}$, $\bm\in\mathrm{MaxMot}(\DC)$. 
It follows that a motivic prime $\p_\cm$ contains $J_{\CC^{[p]}}$ if and only if $\am+\overline{\overline{\bm}^p}$ has at least one component equal to $1$, or, equivalently, such that 
\begin{equation}
\am+\overline{\overline{\bm}^p}<*\dots *\label{eq_star}
\end{equation}
(as $\am\in\{0,*\}^{2n}$ and $\overline{\overline{\bm}^p}\in \{1,*\}^{2n}$) for every $\bm\in\mathrm{MaxMot}(\DC)$.
%and the minimal primes of $J_{\CC^{[p]}}$ are those $\p_\cm$ for which $\cm=\am$ and $\am$ is maximal among the motifs for which (\ref{eq_star}) holds. 
Let $\{i_1,\dots, i_k\}$ be the set of elements $i$ of $[n]$ such that $a_i=a_{i+n}=0$. (This includes the possibi\-li\-ty $k=0$.) Let $\am^{\tu\, d}$ denote $\am^{\tu\, d}_{i_1,\dots, i_k, i_1+n,\dots, i_k+n}$. The statement of the theorem follows if we now justify the claim that $\am$ satisfies (\ref{eq_star}) if and only if each partial word $\wm\le \am^{\tu\, d}$ belongs to $\CC^{^{\tu}}_{i_1,\dots, i_k}$, i.e., if and only if $\am^{\tu\, d}\in\mathrm{ParMot}(\CC^{^{\tu}}_{i_1,\dots, i_k})$. The necessity is clear. For the sufficiency, there would otherwise be a partial word $\wm\le \am^{\tu\, d}$ which is not coming from any partial word in $\CC^{^{\tu}}_{i_1,\dots, i_k}$. Let $\wm^\ast$ be the motif obtained by replacing each $\tu$ in $\wm$ by $*$. Then $\wm^\ast\in\mathrm{Mot}(\DC)$, hence $\overline{\overline{\wm^\ast}^p}\in\mathrm{Mot}(^c(\CC^{[p]}))$. This motif would have stars at all the components at which $\am$ has zeros and ones or stars at all other components. Hence $\am+\overline{\overline{\wm^*}^p}=*\dots *$, contradicting to the asumption that $\am$ satisfies (\ref{eq_star}).
\end{proof}

%\begin{remark}
%Note that in the proof of the previous theorem $k=0$ if and only if $\am\in\mathrm{MaxMot}(\CC^{[p]})$, i.e., $\am=\cm^p$ for some $\cm\in\mathrm{MaxMot}(\CC)$. This confirms again that $\mathrm{Min}^p(\CC)\subseteq \mathrm{Min}(\CC^{[p]})$, i.e., $\mathrm{Min}(\CC^p)\subseteq \mathrm{Min}(\CC^{[p]})$.
%\end{remark}

\begin{example}\label{main_example_cont2}
In the context of the examples (\ref{main_example}) and (\ref{main_example_cont}), 
let $\cm=00\!*\!*0*\in\M^6$. Then $\am=\cm$ and 
\[\am^{\tu\, d}_{2,2}\;=\;0\tu\! *\!*\tu *^d\;=\;0\tu\!*.\]
Also
\[\CC^{^\tu}_2=\{000, \,100, \,110,\, 011\}^{^\tu}_2 = \{0\tu 0, \,1\tu 0,\, 0\tu 1\}.\]
Hence
\[\mathrm{MaxParMot}(\CC^{^\tu}_2) \;=\;\{*\tu 0,\, 0\tu *\}.\]
Thus 
\[\am^{\tu\, d}_{2,2} \in \mathrm{MaxParMot}(\CC^{^\tu}_2),\]
so that 
\[\p_\cm = (X_1, X_2, Y_2) \supseteq J_{\CC^{[p]}}.\]
In fact we have 
\[\p_\cm = (X_1, X_2, Y_2) \in\mathrm{Min}(J_{\CC^{[p]}}).\]

\medskip
Let now $\cm=0\!*\!0\!*\!*0\in\M^6$. Then $\am=\cm$ and 
\[\am^{\tu\, d}_{3,3}\;=\;0\!*\!\tu\!*\!*\tu^d\;=\;0\!*\!\tu.\]
Also
\[\CC^{^\tu}_3=\{000, \,100, \,110,\, 011\}^{^\tu}_3 = \{00\tu, \,10\tu,\, 11\tu, \, 01\tu\}.\]
Hence
\[\mathrm{MaxParMot}(\CC^{^\tu}_3) \;=\;\{*\!*\!\tu\}.\]
Thus 
\[\am^{\tu\, d}_{3,3} \in \mathrm{ParMot}(\CC^{^\tu}_3),\]
so that 
\[\p_\cm = (X_1, X_3, Y_3) \supseteq J_{\CC^{[p]}}.\]
%However,
%\[\am^{\tu\, d}_{3,3} \notin \mathrm{MaxParMot}(\CC^{^\tu}_3),\]
%so 
Note that 
\[\p_\cm = (X_1, X_3, Y_3) \notin \mathrm{Min}(J_{\CC^{[p]}})\]
%This is in accordance with 
since it was shown in Example \ref{main_example_cont} that the prime $(X_3, Y_3)$ is a minimal prime of $J_{\CC^{[p]}}$.

\medskip
Finally, let $\cm=100\!*\!0*\in\M^6$. Then $\am=*00\!*\!0*$ and 
\[\am^{\tu\, d}_{2,2}\;=\;*\tu 0\!*\!\tu*^d\;=\;*\tu 0.\]
Also
\[\CC^{^\tu}_2=\{000, \,100, \,110,\, 011\}^{^\tu}_2 = \{0\tu 0, \,1\tu 0,\, 0\tu 1\}.\]
Hence
\[\mathrm{MaxParMot}(\CC^{^\tu}_2) \;=\;\{*\tu 0, \;0\tu 1\}.\]
Thus 
\[\am^{\tu\, d}_{2,2} \in \mathrm{MaxParMot}(\CC^{^\tu}_2),\]
so that 
\[\p_\am = (X_2, X_3, Y_2) \supseteq J_{\CC^{[p]}}\]
and
\[\p_\cm = (1-X_1, X_2,X_3, Y_2) \supseteq J_{\CC^{[p]}}.\]
%However,
%\[\am^{\tu\, d}_{3,3} \notin \mathrm{MaxParMot}(\CC^{^\tu}_3),\]
%so 
Note that 
\[\p_\am \notin \mathrm{Min}(J_{\CC^{[p]}})\]
even though $\am^{\tu\, d}_{2,2} \in \mathrm{MaxParMot}(\CC^{^\tu}_2)$
%This is in accordance with 
since it was shown in Example \ref{main_example_cont} that the prime $(X_2, X_3)$ is a minimal prime of $J_{\CC^{[p]}}$.

\end{example}

\end{document}